\documentclass[12pt]{article}
\usepackage{amsmath,amssymb,amsthm}

\theoremstyle{definition}
\newtheorem{thm}{Theorem}
\newtheorem{lem}[thm]{Lemma}
\newtheorem{defn}[thm]{Definition}
\newtheorem{prob}{Problem}

\title{Toward an uncountable analogue of Gallai's Theorem for colorings of the plane}
\author{Jeremy F. Alm}
\date{}

\begin{document}
% prints the title

\maketitle

\abstract{In this paper we prove that if $S$ is any finite configuration of points in $\mathbb{Z}^2$, then any finite coloring of $\mathbb{E}^2$ must contain uncountably many monochromatic subsets homothetic to $S$.  We extend a result of Brown, Dunfield, and Perry on 2-colorings of $\mathbb{E}^2$ to any \emph{finite} coloring of $\mathbb{E}^2$.}

% suppresses page numbers

% body of text

  \section{Introduction}

  Let $\mathbb{E}^2$ denote the Euclidean Plane.  Many authors have  considered the question, for which finite sets $S$  in $\mathbb{E}^2$ is it true that if the points of $\mathbb{E}^2$ are colored in finitely many colors, there must be some monochromatic subset congruent to $S$? (For an extensive treatment of this and related problems, see \cite{Soifer}.) Call such a set $S$ \emph{Ramsey (for $\mathbb{E}^2$)}.  In \cite{Shader}, it is proved that all right triangles are Ramsey for two colors. No equilateral triangles are  Ramsey, for we can avoid a monochromatic equilateral triangle with side length $d$ by coloring the plane with  vertical strips of width $\sqrt{3}d/2$, alternating in red and blue.  In \cite{Erdos}, it is conjectured that  if a 2-coloring of $\mathbb{E}^2$ contains no monochromatic equilateral triangles with unit side length, then it contains monochromatic equilateral triangles of \emph{all} other sizes.

We might also consider replacing
the word \emph{congruent} above with \emph{homothetic}
(where $A$ is \emph{homothetic} to $B$ if $A$ can be mapped onto $B$ by a translation and a dilation).  Gallai's Theorem solves this problem in the affirmative---see below.

  Finally, let us further expand our  consideration to include all monochromatic subsets \emph{similar} to $S$ (where $A$ is \emph{similar} to $B$ if $A$ can be mapped onto $B$ by some sequence of translations,  dilations, rotations, and reflections).  In Theorem 4 of \cite{BDP}, the authors show that in any given  2-coloring of $\mathbb{E}^2$ there exist uncountably many $r\in\mathbb{R}^+$ so that there is a monochromatic equilateral triangle with side length $r$; hence there exist uncountably many monochromatic sets similar to any given equilateral triangle.  It is this result that we strengthen in this paper, especially in Theorem \ref{gen} below.

  Throughout, we  let $[k]=\{1,\ldots,k\}$.

  \section{Main Results}

 We begin with a definition.

 \begin{defn}
   A \emph{rectangle} in $\mathbb{E}^2$ is a set  of  four points $$\{(x,y),(x+d_1,y),(x,y+d_2),(x+d_1,y+d_2)\},$$ with $d_1,d_2>0$. A \emph{rectangle} in $\mathbb{Z}^2$ is a set  of  four points $$\{(x,y),(x+d_1,y),(x,y+d_2),(x+d_1,y+d_2)\},$$ with $x,y,d_1,d_2\in \mathbb{Z} $ and $d_1,d_2>0$.   A \emph{square} in $\mathbb{E}^2$ (respectively, $\mathbb{Z}^2$) is a rectangle in $\mathbb{E}^2$ (respectively, $\mathbb{Z}^2$) with $d_1=d_2$.
 \end{defn}

 Note that in this paper, all rectangles and squares have sides parallel to the axes. The proof of the following lemma  is left to the reader.

 \begin{lem} \label{lemma}
   Let $\mathbb{Z}\times\mathbb{Z}$ be colored in finitely many colors; then there exists some monochromatic rectangle.
 \end{lem}

Throughout this paper, we will rely heavily on factoring $\mathbb{E}^2$ into cosets, as in $^{\mathbb{E}^2}/_{\mathbb{Z}\times \mathbb{Z}}$.  Each coset $C\in {}^{\mathbb{E}^2}/_{\mathbb{Z}\times\mathbb{Z}}$ is of the form $$(x,y)+(\mathbb{Z}\times \mathbb{Z}),\qquad(x,y)\in[0,1)\times[0,1),$$ and so $C$ is an infinite grid; hence Lemma \ref{lemma} applies to $C$ (even though $C$ is not equal as a set to $\mathbb{Z}\times \mathbb{Z} $).  This yields the following result:

 \begin{thm}\label{rectangle}
  For each $r\in\mathbb{R}^+$ and for every   finite coloring of $\mathbb{E}^2$ there exist uncountably many monochromatic rectangles with side lengths that are integer multiples of $r$.
 \end{thm}

 \begin{proof}
   Let $r\mathbb{Z}$ denote $\{ rn:n\in\mathbb{Z}\}$.  Then $^{\mathbb{E}^2}/_{r\mathbb{Z}\times r\mathbb{Z}}$ is a collection of cosets of the form  $$(x,y)+(r\mathbb{Z}\times r\mathbb{Z}),\qquad(x,y)\in[0,r)\times[0,r).$$ A fixed coset $C$ is a grid that is a translation of $r\mathbb{Z}\times r\mathbb{Z}$.  By  Lemma \ref{lemma}, $C$ contains a monochromatic rectangle.
 \end{proof}

  ``Gallai's Theorem", which first appeared in the literature in \cite{Rado}, refers to one of two results:

 \begin{thm}[Gallai's Theorem on $\mathbb{Z}^2$] \label{GZ}
   Let $S$ be a finite subset of $\mathbb{Z}^2$.  Then any finite coloring of $\mathbb{Z}^2$ contains a monochromatic subset homothetic to $S$.
 \end{thm}

 \begin{thm}[Gallai's Theorem on $\mathbb{E}^2$]
   Let $S$ be any finite subset of $\mathbb{E}^2$.  Then any finite coloring of $\mathbb{E}^2$ contains a monochromatic subset homothetic to $S$.  \end{thm}

    (For a discussion of Gallai's Theorem as well as a proof, see \cite{Soifer}, p. ~508.)  From Theorem \ref{GZ}, we see that any finite coloring of $\mathbb{Z}\times\mathbb{Z}$, hence any of our coset ``grids", will contain a monochromatic square. Using this result, we can prove a variation on  Theorem \ref{rectangle} in which all of the monochromatic rectangles are similar to one another.  Let the \emph{aspect ratio} of a rectangle denote the ratio of a rectangle's width to its height.

 \begin{thm}\label{aspect}
   For each $r\in\mathbb{R}^+$ and for every   finite coloring of $\mathbb{E}^2$, there exist uncountably many rectangles with aspect ratio $r$.
 \end{thm}

 \begin{proof}
   Consider $^{\mathbb{E}^2}/_{r\mathbb{Z}\times \mathbb{Z}}$.  Any coset $C$ is a grid, and hence by Theorem \ref{GZ} contains a ``square" of the form $\{ (rn,m),(r(n+d),m),(rn,m+d),(r(n+d),m+d)\}$, which corresponds to
a rectangle  with width  $rd$ and height $d$.
 \end{proof}

 In \cite{AxMan}, the authors try to find the smallest $n$ so that any 2-coloring of $[n]\times[n]$ contains a monochromatic square; they show $n\geq 13$.  In \cite{BE}, the authors prove that $n=15$ with the aid of computers.  Hence we may give the following improvement of Theorem \ref{aspect} for two colors:

 \begin{thm}
   Let $0<r<1$.  For all 2-colorings of the unit square there exist uncountably many monochromatic rectangles with aspect ratio $r$.
 \end{thm}

 \begin{proof}
   Let the unit square be 2-colored, and let $0<r<1$.  We need only consider the subset $[0,r]\times[0,1]$ of the unit square. Let $A=\{ 0,\frac{1}{15},\frac{2}{15},\ldots,\frac{14}{15}\}$, and let $rA=\{ 0,\frac{r}{15},\frac{2r}{15},\ldots,\frac{14r}{15}\}$.  Consider the collection of cosets $$^{[0,r]\times[0,1]}/_{rA\times A}$$  Each coset is a 15-by-15 grid, hence by \cite{BE} contains a monochromatic ``square", which is a rectangle with aspect ratio $r$.
 \end{proof}

 Now let us consider equilateral triangles.  In \cite{BDP}, the following appears as  Theorem 4:

 \begin{quote}
For every two[-]coloring of $\mathbb{E}^2$, there exist an uncountable number of values of $r$, where $r\in \mathbb{R}^+$,  such that an equilateral triangle of side $r$ exists monochromatically.
 \end{quote}
 We extend this result to a stronger version that applies to any \emph{finite} coloring of $\mathbb{E}^2$. Let $\mathcal{T}$ be the \emph{unit equilateral triangle} $\{(0, 0),(1,0), (1/2, \sqrt{3}/2)\} $.

 \begin{thm}\label{gen}
 For any finite coloring of $\mathbb{E}^2$ there exist uncountably many $r\in\mathbb{R}^+$ such that there exist uncountably many monochromatic equilateral triangles with side length $r$ homothetic to $\mathcal{T}$.  Furthermore, for any 2-coloring of the unit square there exist uncountably many $r\in(0,1]$ such that there exist uncountably many monochromatic equilateral triangles of side length $r$ homothetic to $\mathcal{T}$.
 \end{thm}

    %$\langle (r, 0),(r/2,r\sqrt{3}/2)\rangle$
 \begin{proof}
    For the proofs of both claims, we will consider two copies of $\mathbb{E}^2$, which we will think of  as vector spaces---the first copy (the domain) with respect to the basis $\langle (1, 0), (0, 1)\rangle$, and  the second copy (the range) with respect to the basis $\langle(r,0),(\frac{r}{2},\frac{r\sqrt{3}}{2})\rangle$, with $r$ to be chosen later. We will also need the vector-space isomorphism $\varphi:\mathbb{E}^2\to\mathbb{E}^2$ that maps
   \begin{align*}
     (1,0) &\mapsto (r,0)\\
     &\text{and}\\
     (0,1) &\mapsto \left(\frac{r}{2},\frac{r\sqrt{3}}{2}\right).
   \end{align*}
 Notice that $\varphi$ sends a \emph{square} to a \emph{rhombus}.

   To prove the first claim, fix $r\in\mathbb{R}^+$ and suppose the range $\mathbb{E}^2$ is finitely colored by $\chi:\mathbb{E}^2\to [k]$.  Let $\chi$ induce a coloring $\chi' $ of the domain $\mathbb{E}^2$ via $\varphi^{-1}$, i.e. let $\chi' (x,y)=\chi(\varphi(x,y))$.    Each coset $C\in {}^{\mathbb{E}^2}/_{\mathbb{Z}\times\mathbb{Z}}$ contains a monochromatic square by Theorem \ref{GZ}.  The image of such a square under $\varphi$ is a rhombus with side length $rk$ for some $k\in\mathbb{Z}^+$, and acute interior angles of measure 60 degrees.  This gives  a monochromatic equilateral triangle with side length an integer multiple of $r$ homothetic to $\mathcal{T}$.  Letting $r$ range over $\mathbb{R}^+$ gives the desired result.

   To prove the second claim, let $0<r\leq\frac{2}{45}$ and let the unit square $[0,1]\times[0,1]$ be 2-colored by $\chi:[0,1]\times[0,1]\to [2]$.  Now $\varphi([0,15]\times[0,15])$ is a subset of the unit square;
   let $\chi$ induce a 2-coloring $\chi' $ of $[0,15]\times[0,15]$ via $\varphi^{-1}$, where $\chi' (x,y)=\chi(\varphi(x,y))$. Let $A=\{ 0,1,2,\ldots,14\}$. By \cite{BE}, any coset in $^{[0,15]\times[0,15]}/_{A\times A}$, which is  a 15-by-15 grid, contains a monochromatic square.  As before, the image under $\varphi$ of such a square contains a monochromatic equilateral triangle with side lengths (small) integer multiples of $r$ homothetic to $\mathcal{T}$.  Now letting $r$ range over uncountably many values in $(0,\frac{2}{45}]$ that are linearly independent over $\mathbb{Z}$ gives the desired result.
 \end{proof}

 %\begin{proof}
%   Let $0<r\leq \frac{2}{45}$.  Consider $\mathbb{E}^2$ with respect to the basis $\{(r,0),(\frac{r}{2},\frac{r\sqrt{3}}{2})\}$.  The span of this basis using only scalars from $\mathbb{Z}$ puts a grid on $\mathbb{E}^2$; since $r\leq 1$, the rectangle $[0,15]\times[0,16]$ contains a subgrid of size at least 15-by-15.  Consider the vector-space isomorphism $\varphi$ from $\mathbb{E}^2$ to $\mathbb{E}^2$ that maps

%
%   Let the coloring of the range $\mathbb{E}^2$ induce a coloring of the domain $\mathbb{E}^2$ via $\varphi^{-1}$.  The image of $[0,15]\times[0,15]$ under $\varphi$ is a subset of $[0,15]\times[0,16]$.  Form cosets of the domain by $\mathbb{Z}\times\mathbb{Z}$, partitioning $[0,15]\times[0,15]$ into $[0,1)^2$-many 15-by-15 grids.  By \cite{BE}, each of these grids contains a monochromatic square with side length $k$, $k\in[15]$.  The image of such a square under $\varphi$ is a rhombus with side length $rk$ and altitude $rk$.  This gives us a monochromatic equilateral triangle with side length a (small) integer multiple of $r$.  Now letting $r$ range over uncountably many values from $(0,\frac{1}{15}]$ that are linearly independent over $\mathbb{Z}$ gives the desired result.
% \end{proof}

 Let us generalize and summarize what we have done so far:

 \begin{thm}
   Let $S$ be a finite configuration of points in the integer lattice $\mathbb{Z}\times\mathbb{Z}$.  In any finite coloring of the plane $\mathbb{E}^2$, there exist uncountably many monochromatic homothetic copies of $S$.
 \end{thm}

 \begin{proof}
    Consider the collection of cosets $^{\mathbb{E}^2}/_{\mathbb{Z}\times \mathbb{Z}}$. A fixed coset $C$ is a translation of $\mathbb{Z}\times\mathbb{Z}$. By Gallai's theorem, each coset contains a monochromatic subset homothetic to $S$.
 \end{proof}

\section{Conclusions}
The results given here rely on partitions of the plane into ``nice" cosets, or cosets skewed by a linear transformation.  The proof of Theorem \ref{gen} could be adapted to any rhombus or parallelogram, hence any triangle.

\begin{prob}
  Show that for any finite coloring of $\mathbb{E}^2$ and any  4-point configuration in the plane, some color class must contain uncountably many homothetic copies of the configuration.
\end{prob}

It seems natural to conjecture that any finite configuration in the
plane---which must appear in some color class by
Gallai's Theorem on $\mathbb{E}^2$---must appear in fact uncountably many times.

\begin{prob}
  Show that for any finite coloring of $\mathbb{E}^2$ and any finite $S\subset\mathbb{E}^2$, some color class must contain uncountably many homothetic copies of $S$.
\end{prob}A solution to either of these problems would presumably require a partition of the plane more clever than the ones given here.

\section{Acknowledgements}
The author wishes to thank Jacob Manske for uncountably many productive conversations, and for reading several drafts of this paper.

%\bibliographystyle{abbrv}
%\bibliography{refs}

\end{document}